\documentclass[reqno]{amsart}                             

\usepackage{amsmath}
\usepackage{amssymb}
\usepackage{amsfonts}

\usepackage{cite}
\usepackage{amsmath,amssymb,amsfonts}
\usepackage{algorithmic}
\usepackage{graphicx}
\usepackage{textcomp}
\usepackage{cancel}
\usepackage{xcolor}
\usepackage{pdfcomment}
\usepackage[tikz]{ocgx2}
\usepackage{marginnote}

\usepackage[normalem]{ulem}

\usepackage{graphicx}
\usepackage{subfigure}

\usepackage{epsfig} 
\usepackage{cite}

\usepackage{bm}

\usepackage{hyperref}
\usepackage{cleveref}

\usepackage[normalem]{ulem}
\usepackage{comment}
\usepackage{enumerate}
\usepackage{mathrsfs}

\newtheorem{theorem}{Theorem}[section]

\newtheorem{proposition}[theorem]{Proposition}
\newtheorem{corollary}[theorem]{Corollary}

\newtheorem{assumption}[theorem]{Assumption}

\newtheorem{thm}{Theorem}[section]

\newtheorem{prop}[thm]{Proposition}
\newtheorem{rem}[thm]{Remark}

\newcommand{\mytilde}{\raise.17ex\hbox{$\scriptstyle\mathtt{\sim}$}}

\newcommand{\RgeO}{\ensuremath{\mathbb{R}_{ \geq 0}}}
\newcommand{\Rat}[1]{\ensuremath{\mathbb{R}^{#1}}}

\newcommand{\bone}{\mathbf{1}}
\newcommand{\A}{\mathcal{A}}
\newcommand{\B}{\mathcal{B}}

\renewcommand{\P}{\mathcal{P}}

\newcommand{\X}{\mathcal{X}}
\newcommand{\bP}{\mathbb{P}}

\newcommand{\argmin}{{\rm argmin}}

\newcommand{\class}{\mathcal{C}}

\newcommand{\Ball}[2]{\mathcal{B}_p(#1,#2)}

\newcommand{\empirical}{P^n}

\usepackage{xparse}

\NewDocumentCommand{\lcAmbSet}{g g g g}
{
\IfNoValueTF{#1}
  {\mathcal{T}_{p,\A_m}}
  {
    \IfNoValueTF{#3}
      {\mathcal{T}_{p,\A_m}(#1,#2)}
      {
        \IfNoValueTF{#4}
          {\mathcal{T}_{#1,\A_m}(#2,#3)}
          {\mathcal{T}_{#1,#2}(#3,#4)}
      }
  }
}

\newcommand{\bE}{\ensuremath{\mathbb{E}}}
\newcommand{\eps}{\ensuremath{\varepsilon}}

\newcommand{\longthmtitle}[1]{\mbox{}{\bf \textit{(#1).}}}

\usepackage{todonotes}

\allowdisplaybreaks

\title{Local-discrepancy ambiguity sets for distributionally robust optimization}

\author[I. Papagoudis]{Ilias Papagoudis}
\address{Delft Center for Systems and Control, Faculty of Mechanical Engineering, TU Delft}
\email{i.p.papagoudis@tudelft.nl}
\author[D. Boskos]{Dimitris Boskos}
\address{Delft Center for Systems and Control, Faculty of Mechanical Engineering, TU Delft}
\email{d.boskos@tudelft.nl}

\begin{document}

\thanks{Delft Center for Systems and Control, Faculty of Mechanical Engineering, TU Delft}
\thanks{This work was supported by NWO grant OCENW.M.22.119}


\begin{abstract}
In this paper, we introduce a new  class of data-driven ambiguity sets for distributionally robust optimization (DRO) that aim to enhance the local inference accuracy of an unknown distribution. Building on optimal transport–based discrepancy measures, we depart from the conventional paradigm of monolithic ambiguity balls, which rely on a single global transport budget to capture distributional uncertainty. Instead, we introduce local-discrepancy ambiguity sets, where the support of the distribution is partitioned into predefined regions and transport cost constraints are imposed locally to each region. This added flexibility allows for a more refined characterization of plausible distributions by preventing highly imbalanced transport across regions. As a result, the proposed approach can yield less conservative and more robust decisions, particularly in settings where decisions depend on the local properties of the sample space. We further derive tractable reformulations of the resulting DRO problems and demonstrate through simulations that our method can outperform traditional Wasserstein ambiguity sets in terms of decision quality.
\end{abstract}

\maketitle

\section{Introduction}
\label{sec:introduction}

As modern systems become increasingly complex, their level of uncertainty continues to grow. Consequently, optimization methods that assume perfect model information are often inadequate for addressing decision-making problems in such settings. Optimization under uncertainty provides a more suitable framework for these challenges.
A central approach in this context is stochastic optimization, which determines optimal decisions for problems affected by randomness, typically in their cost functions and constraints. A common standing assumption is that the underlying random variables follow a known probability distribution \cite{shapiro2021lectures}. While this approach is powerful in many cases, it relies on the strong premise that the underlying uncertainty is fully characterized.

In practice, one often has access only to samples from the underlying distribution rather than the distribution itself, and must balance robustness and performance based on them. A classical data-driven decision-making approach is the sample average approximation (SAA), which replaces the true distribution with the empirical distribution generated by the samples \cite{shapiro2021lectures}. This method is computationally tractable and enjoys asymptotic convergence guarantees~\cite{BirgeLouveaux2011,king1991epi,kleywegt2002sample}, but lacks finite-sample guarantees and may yield over-optimistic decisions due to the optimizer’s curse \cite{smith2006optimizer}.

Distributionally robust optimization (DRO) addresses these issues by constructing an ambiguity set of plausible probability distributions based on the data and optimizing against the worst-case expectation over this set. This framework provides robustness with probabilistic guarantees while maintaining good performance, and has been successfully applied in areas such as control \cite{cherukuri2019cooperative,shafieezadeh2018wasserstein,yang2020wasserstein}, energy systems \cite{wang2018risk}, and machine learning \cite{abdullah2019wasserstein,najafi2019robustness}. A key challenge in DRO is the design of the ambiguity set: overly large sets lead to conservative solutions, while overly small ones fail to ensure robustness, requiring a careful trade-off between the two.

Ambiguity sets can be broadly classified into moment-based and discrepancy-based sets. Moment-based sets impose constraints on distributional moments \cite{popescu2007robust,delage2010distributionally}, offering computational tractability, albeit with weaker statistical guarantees and some reliance on chosen moment constraints. Discrepancy-based sets, or ambiguity balls, include distributions within a prescribed distance from a nominal distribution, defined through discrepancies such as $\phi$-divergences \cite{jiang2016data}, the total variation \cite{sun2016convergence}, the Prokhorov metric \cite{erdougan2006ambiguous}, and the Wasserstein distance \cite{zhao2018data,mohajerin2018data}. Among these, Wasserstein balls provide finite-sample guarantees of containing the true distribution when centered at the empirical distribution \cite{fournier2015rate}, accommodate discrete and continuous distributions, and capture support mismatches, unlike divergences or the total variation metric.

For practical implementation, tractable reformulations of the Wasserstein DRO problem are essential, as worst-case expectations are taken over infinite-dimensional ambiguity sets. Such reformulations exist for broad classes of cost functions \cite{mohajerin2018data,gao2023distributionally,blanchet2019quantifying}, yielding equivalent problems that can be efficiently solved with standard optimization methods.
When the goal is to solve a single optimization problem, the ambiguity set can also be informed by the problem itself, enabling high-confidence optimality guarantees without requiring that the ambiguity set contain the true distribution with the same confidence~\cite{JB:19,MS-PP:23,RG:23}.
In more general settings, one may be interested in solving multiple optimization problems under the same uncertainty. In such cases, despite their widespread success, Wasserstein ambiguity sets may contain highly irrelevant distributions, which can in turn lead to conservative decisions. By highly irrelevant, we mean that the candidate distributions are highly unlikely to have produced the observed data and may significantly affect the expected value of the objective function.

To address this challenge, we introduce in this paper a new type of ambiguity set, which excludes such irrelevant elements, improving the performance and reliability of Wasserstein-based DRO algorithms. Our guiding intuition is that for uniformly distributed uncertainty and sufficiently many samples, the transport cost between the empirical and the true underlying distribution will also be close to uniform across the sample space. Thus, by partitioning the sample space into several smaller regions, we can better control the fluctuations between the empirical and the true distribution from region to region.

Our main contribution is the introduction of a vector-valued optimal-transport-based discrepancy, where each component controls the optimal transport budget from a corresponding region of the reference distribution. Based on this new discrepancy, we construct an ambiguity set, coined a local-discrepancy ambiguity set, which includes all distributions close to the reference distribution, in the sense that they respect all the optimal transport constraints encoded by the discrepancy vector.

Our second contribution is the derivation of strong duality results for local-discrepancy ambiguity sets for a broad class of cost functions. Our third contribution is to obtain tractable reformulations for certain classes of cost functions, including reformulations for optimal facility location problems. Our last contribution is the introduction of a computational calibration procedure that selects the ambiguity-set radii to achieve a prescribed empirical coverage level over a given distribution class. This provides a common basis for comparing local-discrepancy sets with Wasserstein balls at matched confidence levels and facilitates a fair numerical comparison between them. Analytical finite-sample guarantees for the proposed ambiguity sets require a separate statistical analysis, rather than a direct application of standard Wasserstein concentration results, and are beyond the scope and space limitations of the present conference paper.

\section{Preliminaries and Notation}
\label{sec:preliminaries}
Throughout the paper, we use the following notation.
We denote by $\mathbb{R}_{\geq0}$ and $\mathbb{R}_{>0}$ the non-negative and positive real numbers, respectively, and define $\Bar{\mathbb{R}}=\mathbb{R}\cup\{-\infty,\infty\}$.
We use the notation $\preceq$ for the elementwise partial order in $\Rat{n}$ and denote by $\#A$ the cardinality of a set $A$. For $m\in\mathbb{N}\setminus\{0\}$, we denote $[m] := \{1, \ldots, m\}$. 
We denote by ${\rm int}(A)$, ${\rm cl}(A)$, and ${\rm bdry}(A)$, respectively, the interior, closure, and boundary of a set $A$ in some topological space. Given $\bm\eps=(\eps_1,\ldots,\eps_m)\in\RgeO^m$, we denote by $[0,\bm\eps]$ the hyperrectangle $\prod_{k=1}^m[0,\eps_k]\subset\Rat{m}$. The first and second projection on the product space $X\times Y$ are ${\rm{pr_1}}(x,y):=x$ and ${\rm pr_2}(x,y):=y$, respectively.  

Given the measurable spaces $(\Omega,\mathcal F)$ and $(\Omega',\mathcal F')$, a measurable map $\Psi: (\Omega,\mathcal F) \to (\Omega',\mathcal F')$ assigns to each measure $\mu$ in $(\Omega,\mathcal F)$ the pushforward measure $\Psi_\#\mu$ in $(\Omega',\mathcal F')$ defined by $\Psi_\#\mu(B):=\mu(\Psi^{-1}(B))$ for all $B\in\mathcal F'$.
We denote the class of distributions on a Polish space $\Xi$ (separable metric space) as $\P(\Xi)$, and we will always assume $\Xi$ equipped with its Borel $\sigma$-algebra $\B(\Xi)$. 
We denote by $\Pi(P,Q)$ the set  of couplings between two distributions  $P$ and $Q$ in $\P(\Xi)$, i.e., distributions $\pi\in\P(\Xi\times\Xi)$ with marginals  ${\rm pr}_{1\#}\pi=P$ and ${\rm pr}_{2\#}\pi=Q$, respectively. We denote by $\Pi(Q)$ all the distributions on $\Xi\times\Xi$ with $Q$ as their first marginal.
Given a metric $\rho$ on $\Xi$ and two distributions $P,Q\in\P_p(\Xi)$ (with finite $p$th moment), their $p$th Wasserstein distance is 
\begin{align*}
W_p(Q,P):= \inf_{\pi\in\Pi(Q,P)}\Big(\int_{\Xi\times\Xi}\rho(x,y)^pd\pi(x,y)\Big)^{1/p}.
\end{align*}
%

We denote by $C(\Xi)$ the class of continuous real-valued functions on a topological space $\Xi$, and by $C_{2,{\rm const}}(\Xi\times\Xi)$ the functions $\gamma \in C(\Xi\times\Xi)$  with  $\gamma(\zeta,\xi) =\gamma(\zeta,\xi') $ for all $\zeta,\xi,\xi'\in\Xi$.
A function $h: \Rat{n} \to \bar{\Rat{}}$ is proper if $h(\xi) < +\infty$ for at least one $\xi$ and $h(\xi) > -\infty$ for every $\xi\in\Rat{n}$.  The conjugate of $h$ is defined as $h^*(z) := \sup_{\xi\in\Rat{n}} \langle z,\xi \rangle - h(\xi)$. 
The characteristic function of a set $A\subset\Rat{d}$ is $\chi_{A}(\xi) := 0$ if $\xi\in A$ and $\infty$ otherwise.
Its conjugate function is the support function of $A$ and it is equal to $\sigma_{A}(z) := \sup_{\xi\in A} \langle z,\xi\rangle$. The indicator function of  $A$ is $\bone_{A}(\xi):= 1$ if $\xi\in A$ and $0$ otherwise. We denote by $B(x,r)$ the ball with center $x$ and radius $r$, and by $\mathbb S^{d-1}$ the unit sphere in $(\Rat{d},\|\cdot\|_2)$. 

\section{Problem formulation}

Stochastic optimization revolves around optimal decision-making for systems that are affected by stochastic uncertainty. In the ideal case where the distribution of the uncertainty is known, a stochastic optimization problem takes the form  
\begin{align*}
    \inf_{x\in \mathcal{X}} \mathbb{E}_{P}[f(x,\xi)]
\end{align*}
where the cost function $f:\mathcal{X}\times \Xi \to \mathbb{R}$ depends on the decision variable $x\in\mathcal{X}$ and the uncertainty $\xi$, which takes values in the Polish space $\Xi$ with metric $\rho$ and has probability distribution~$P$. In this setup, one seeks to solve the optimization problem to determine the best decision on average. 

However, in real-world applications, the uncertainty $P$ is typically unknown. What is usually provided instead are $n$ i.i.d. samples $\xi_1,\ldots,\xi_n$ following the law $P$, with the help of which we aim to approximate the true distribution.
In this case, instead of the original optimization problem one can consider the so-called sample average approximation (SAA), which evaluates the expected cost over the empirical distribution $\empirical:= \frac{1}{n}\sum_{i=1}^n \delta_{\xi_i}$. Here each $\delta_{\xi_i}$ denotes the Dirac measure at point $\xi_i$. Since the SAA  relies directly on the data, its performance is highly dependent on their quality and quantity. When the data are limited or noisy, SAA may exhibit poor out-of-sample performance.

To address such issues, distributionally robust optimization (DRO) has emerged as a decision-making framework that hedges against uncertainty in the uncertainty. The key idea is to optimize over the worst-case distribution from an ambiguity set $\P$ of plausible distributions, which depends on the provided data. Namely, we formulate the DRO problem 
\begin{align}
    \inf_{x\in\X} \sup_{P\in \P} \mathbb{E}_P[f(x,\xi)].
    \label{eq:DRO}
\end{align}

The choice of the ambiguity set $\P$ is a key aspect of DRO, since it determines how plausible distributions are grouped; this has implications on the robustness, performance, and tractability of the chosen decision scheme. A widely adopted choice of  ambiguity sets is Wasserstein balls 
\begin{align*}
    \Ball{Q}{\eps}:= \{P\in\P_p(\Xi): W_p(Q,P)\le \eps\},
\end{align*}
which group all distributions up to some distance $\eps$ in the $p$th Wasserstein metric from a reference distribution $Q\in\P_p(\Xi)$. For data-driven problems, the ambiguity balls are typically centered at the empirical distribution $Q\equiv\empirical$. 

However, monolithic ambiguity balls may become conservative, because the distance between two distributions is defined through the global (optimal) cost to transport mass from one to the other, which gives no information about their closeness in sub-regions of their domain (cf. Figure~\ref{fig:motivation}, right).
For instance, a designer may often be interested in quantifying a certain part of the uncertainty. Considering a monolithic Wasserstein ball can lead to the exploitation of the transport budget that captures the dissimilarity in farther away regions, which may, in turn, generate excessive distribution fluctuations in the region of interest. 

These considerations motivate the problem that we aim to address in this paper, which is to \textit{build ambiguity sets that capture the dissimilarity between distributions locally across their domain, using some vector-valued discrepancy measure. These ambiguity sets should yield DRO problems with tractable finite-dimensional formulations. We also seek to compare these sets to monolithic ambiguity balls in data-driven settings.}

\begin{figure}[h]
	\centering
	\includegraphics[width=0.65\linewidth]{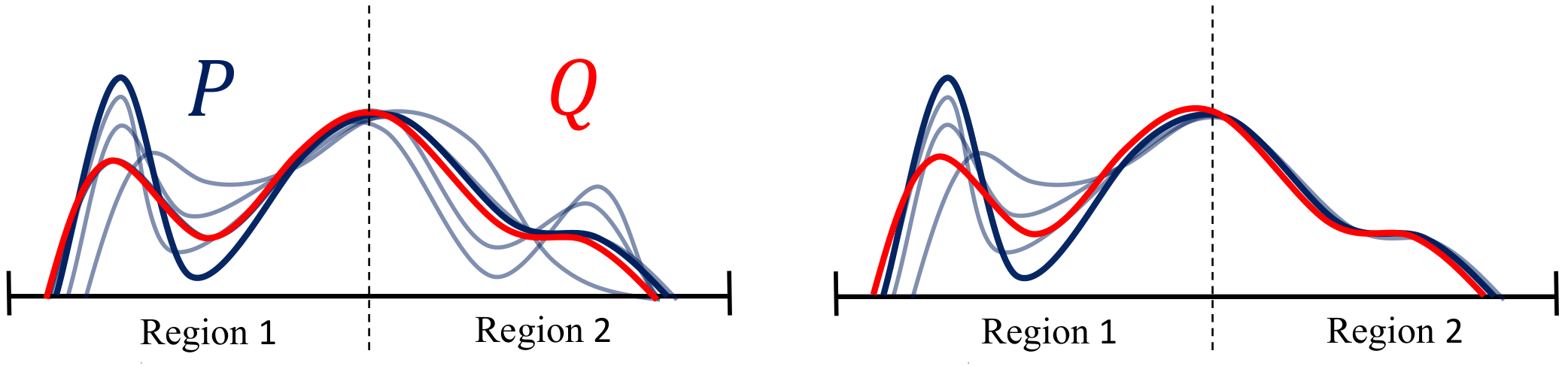} \\ 
	\caption{(Left) The Wasserstein ball around the baseline model $Q$ contains the true distribution $P$ but ignores the prior information that $P$ and $Q$ mostly differ in Region 1. (Right) An  ambiguity set using this local information excludes multiple non-relevant distributions.} \label{fig:motivation}
\end{figure}

\section{Local discrepancy ambiguity sets}

Here, we introduce the desired class of ambiguity sets that can capture local dissimilarities between distributions. To this end, we define a vector-valued discrepancy notion between distributions, which splits the transport cost from one to the other into local costs across their domain. We therefore consider a covering of the domain and determine the individual corresponding costs of transporting mass from the source to the target distribution. 

Consider a covering $\A_m = \{A_k\}_{k=1}^m$ of $\Xi$, which we assume fixed  throughout the paper, consisting of sets with nonempty and mutually disjoint interiors.  For two distributions $P$, $Q$ on $\Xi$, the set of \textit{local feasible vector discrepancies} between $Q$ and $P$ is given by
\begin{align*}
 & T_{p,\A_m}(Q,P):=\{(\eps_1,\ldots,\eps_m)\in\RgeO^m:\;\textup{there exists}\;\pi\in\Pi(Q,P) \\ 
 &\hspace{16.9em}\textup{s.t.}\;\int_{A_k\times \Xi}\rho(\zeta,\xi)^p d\pi(\zeta,\xi)\leq\eps_k \;\textup{for all}\;k\in[m]\Big\}.
\end{align*}
We will also refer to $T_{p,\A_m}(Q,P)$ as a \textit{vector discrepancy} for brevity, and to each of its elements as a \textit{local transport cost vector} between $Q$ and $P$.

Based on the local feasible vector discrepancy, we introduce a new ambiguity set. It comprises all distributions $P$ for which there exists a local transport cost vector between a reference distribution $Q$ and $P$, which respects a vector-valued transport budget $\bm\eps$. To this end, we consider the set of couplings 
\begin{align*}
    \Pi_{p,\A_m}(Q,\bm{\eps}) = \Big\{\pi\in\Pi(Q):\int_{A_k\times \Xi} \rho(\zeta,\xi)^p d\pi(\zeta,\xi)\leq\eps_k \;\textup{for all}\;k\in[m]\Big\}
\end{align*}
and define the \textit{local discrepancy ambiguity set}
\begin{align*}
\lcAmbSet{Q}{\bm\eps} :={\rm{pr}}_{\#2}\Pi_{p,\A_m}(Q,\bm{\eps}),
\end{align*}
which consists of all the second marginals of the couplings in $\Pi_{p,\A_m}(Q,\bm{\eps})$. The following result provides basic properties of local discrepancy ambiguity sets and equivalent criteria to establish that a distribution belongs to such a set.

\begin{prop}
\label{prop:lcAmbSet certif}
\longthmtitle{Local discrepancy ambiguity set containment certificate}
For any $P\in\P_p(\Xi)$ and $\bm\eps\in\RgeO^m$, the ambiguity set $ \lcAmbSet{Q}{\bm\eps}$ is convex and 
\begin{align}
P\in \lcAmbSet{Q}{\bm\eps} \iff
T_{p,\A_m}(Q,P)\cap[0,\bm\eps]\ne \emptyset. \label{local:discrepnacy:certificate}
\end{align}
Let also $\lambda_\star$ be the optimal value of the problem  
\begin{align*}
& \inf_{\lambda \geq 0,\pi\in \Pi(Q,P)} \lambda \\
&\qquad\;\; {\rm s.t. } \int_{A_k \times \Xi} \rho(\zeta,\xi)^p d\pi(\zeta,\xi) \leq \lambda\,\eps_k\;\textup{for all}\;k\in[m].
\end{align*}
Then 
\begin{align*}
        P\in \lcAmbSet{Q}{\bm\eps} & \Longrightarrow  \lambda_\star\le 1 \\
        \textup{and}\;\lambda_\star<1 & \Longrightarrow 
        P\in \lcAmbSet{Q}{\bm\eps}.
    \end{align*}    
\end{prop}

\begin{proof}
Convexity follows directly by linearity of the constraints with respect to the transport plans. The certificate~\eqref{local:discrepnacy:certificate} is also a rather direct consequence of the definitions of the ambiguity set and the vector discrepancy.

To obtain the first implication regarding $\lambda_\star$, let $P\in \lcAmbSet{Q}{\bm\eps}$. Then there exists $\pi \in \Pi(Q,P)$ such that $\int_{A_k \times \Xi}\rho(\zeta,\xi)^p d\pi(\zeta,\xi)\leq\eps_k$ for all $k\in[m]$. This implies  that $\lambda=1$ is feasible for the certificate optimization problem and therefore $\lambda_\star \le 1$. 

For the other implication, let  $\lambda_\star<1$. Then there  exists a feasible $\lambda\le 1$ and a corresponding feasible transport plan $\pi\in \Pi(Q,P)$ with $\int_{A_k \times \Xi}\rho(\zeta,\xi)^p d\pi(\zeta,\xi)\leq \lambda \eps_k\leq\eps_k$ for all $k\in[m]$ This implies that $P\in \lcAmbSet{Q}{\bm\eps}$ and concludes the proof.
\end{proof}

This characterization gives us also an optimization-based criterion to find the smallest size of a local discrepancy ambiguity set that will contain a desired distribution, when we fix the relative magnitudes of the local transport  budgets, i.e., when we require that $\bm\eps=
\lambda\bm v$ for some fixed direction vector $\bm v\in\mathbb S^{m-1}\cap\RgeO^m$.      

\begin{corollary} \label{corollary:minimum:size}
\longthmtitle{Minimum-size local discrepancy ambiguity set of fixed shape}
Consider the direction vector $\bm v\in\mathbb S^{m-1}\cap\RgeO^m$ and let $\lambda_\star\equiv\lambda_\star(\bm v;Q,P)$ be the solution of the optimization problem  
\begin{align*}
& \inf_{\lambda \geq 0,\pi\in \Pi(Q,P)} \lambda \\
&\qquad\;\; {\rm s.t. } \int_{A_k \times \Xi} \rho(\zeta,\xi)^p d\pi(\zeta,\xi) \leq \lambda v_k\;\textup{for all}\;k\in[m].
\end{align*}
Then $P\in \lcAmbSet{Q}{\bm\eps}$ for all $\bm\eps=\lambda\bm v$ with  $\lambda>\lambda_\star$ and $P\notin \lcAmbSet{Q}{\bm\eps}$ for all $\bm\eps=\lambda\bm v$ with $\lambda<\lambda_\star$.
\end{corollary}

\section{Tractable reformulations}

Here we obtain tractable reformulations of DRO problems over local-discrepancy ambiguity sets.
By fixing the decision variable $x\in\X$ in the DRO problem \eqref{eq:DRO} and denoting $h(\xi)\equiv f(x,\xi)$, the key step towards its reformulation is to transform the worst-case expectation 
\begin{align}
 \mathcal I^\star:= \sup_{P\in \P}\mathbb{E}_P[h(\xi)]
    \label{worst case expectation}
\end{align}
into a finite-dimensional optimization problem using duality. To achieve this, we exploit the fact that local-discrepancy ambiguity sets are formulated via a finite set of inequality constraints over the space of probability distributions on $\Xi\times\Xi$, and leverage the strong duality results developed in \cite{chaouach2023structured} for this type of problem.  
We next present the corresponding results and assumptions from \cite{chaouach2023structured}, which considers $m$ different  cost functions $c_k:\Xi\times\Xi\to\mathbb{R}_{\ge0}$, $k\in[m]$.  

\begin{assumption}
	\label{assumption:h}
\longthmtitle{Objective function class} 
The objective function $h:\Xi\to\mathbb{R}$ is upper semicontinuous and integrable with respect to the distribution $Q$, namely, $h\in L^1(Q)$.
\end{assumption}

\begin{assumption}
\label{assumption:cost:functions}
\longthmtitle{Transport costs} 
(i) For each $k\in[m]$  there exists a nondecreasing sequence
$c_{k,\nu}:\Xi\times\Xi\to\mathbb{R}_{\ge0}$, $\nu\in\mathbb{N}$ of \textit{continuous} transport cost functions with $c_{k,\nu}(\zeta,\zeta)=0$ for all $\zeta$ and $c_{k,\nu}(\zeta,\xi)\to c_k(\zeta,\xi)\in\mathbb{R}_{\ge0}$. 

\noindent (ii) There exists a compact set $\Xi_{\rm cmp}\subset\Xi$ such that for any $\nu$, the functions $c_{k,\nu}$, $k\in[m]$ are linearly independent in $C(\Xi_{\rm cmp}\times\Xi_{\rm cmp})$ and $\text{span}\{c_{1,\nu},\ldots,c_{m,\nu}\}\cap C_{2,\text{const}}(\Xi_{\rm cmp}\times\Xi_{\rm cmp})=\{0\}$.
\end{assumption}

\begin{thm}
\label{thm:strong:duality}
\longthmtitle{DRO dual over ambiguity hyperrectangles \cite[Theorem 6.4]{chaouach2023structured}} 
Let $h$ and $c_1,\ldots,c_m$ satisfy Assumptions \ref{assumption:h} and \ref{assumption:cost:functions}, and consider the problem \eqref{worst case expectation} 
with $\P\equiv {\rm pr}_{2\#}\Pi(Q,\bm\epsilon)$, where 
$\Pi(Q,\bm\epsilon):=\{\pi\in\Pi(Q):\int_{\Xi\times\Xi} c_k(\zeta,\xi)d\pi(\zeta,\xi)\le\epsilon_k\;\textup{for all}\;k\in[m]$\}. Then 
\begin{align*}
   \mathcal I^\star=\inf_{\bm\lambda\succeq0}\Big\{\langle\bm\lambda, \bm\epsilon\rangle+\int_{\Xi}\sup_{\xi\in\Xi}\{h(\xi)-\langle \bm\lambda,\bm c(\zeta,\xi) \rangle \}dQ(\zeta)\Big\}.
\end{align*}
\end{thm}
Using these reformulations, we also obtain duality results for DRO problems over local-discrepancy ambiguity sets.
For technical convenience, we \textit{henceforth restrict our attention to the case where $\Xi$ is a closed subset of $\Rat{d}$ satisfying ${\rm cl}_{\Rat{d}}({\rm int}_{\Rat{d}}(\Xi))=\Xi$} and equipped with the restriction of a metric $\rho$ on $\Rat{d}$. Here the $\Rat{d}$-subscript signifies that the topological operations are taken with respect to the topology of $\Rat{d}$ and $\Xi$ viewed as a subset of it. We further assume that \textit{$\Rat{d}$ with the metric $\rho$ is proper, i.e., every closed and bounded subset is compact, and that $\rho$ induces the standard topology on $\Rat{d}$}. As before, we consider a covering $\A_m$ of $\Xi$ by sets with nonempty and mutually disjoint interiors. We also need the further assumption that the reference distribution has zero measure on the boundaries of the sets comprising the covering. This holds always for absolutely continuous distributions on $\Rat{d}$, and almost always for their empirical distributions, which lie in the interior of the covering with probability one, assuming no ``pathological" boundaries of positive Lebesgue measure.

\begin{assumption} \label{assumption:boundary:mass}
\longthmtitle{Probability mass on covering boundary} 
The reference distribution $Q$ satisfies $Q({\rm bdry}(A_k))=0$ for each $k\in[m]$.  
\end{assumption}

\begin{thm} \label{thm:dual:local:discrepancy}

\longthmtitle{DRO dual over local-discrepancy ambiguity sets} 
Under Assumptions~\ref{assumption:h} and~\ref{assumption:boundary:mass}, the worst-case expectation \eqref{worst case expectation} over the local-discrepancy ambiguity set $\lcAmbSet{Q}{\bm\eps}$ is equal to the value of the dual problem 
\begin{align*}
\inf_{\bm\lambda\succeq0}\Big\{
\langle \bm\lambda,\bm\eps \rangle +\sum_{k=1}^m\int_{A_k}\sup_{\xi\in\Xi}\big\{  h(\xi)-
\lambda_{k}\rho(\zeta,\xi)^p\big\}  dQ(\zeta)\Big\}.
\end{align*}
In particular, when $Q\equiv P^n$, we have the dual reformulation  
\begin{align*}
\inf_{\bm\lambda\succeq0}\Big\{\langle \bm\lambda,\bm\eps \rangle+\frac{1}{n}\sum_{k=1}^m\sum_{i=1}^{n_k}\sup_{\xi\in\Xi}\big\{  h(\xi)-
\lambda_{k}\rho(\xi_{ki},\xi)^p \big\}\Big\},
\end{align*}
where $n_k$ is the number of samples in region $A_k$ and $\xi_{ki}$ is the $i$th sample therein.
\end{thm}

\begin{proof}
Note that the local discrepancy ambiguity set has the form $\lcAmbSet{Q}{\bm\eps}\equiv{\rm pr}_{2\#}\Pi(Q,\bm{\eps})$, with $\Pi(Q,\bm\eps)$ as given in the statement of Theorem~\ref{thm:strong:duality} and the costs $c_k(\zeta,\xi) :=\bone_{A_k\times \Xi}(\zeta,\xi)\rho(\zeta,\xi)^p$. Taking further into account that the reference distribution $Q$ is supported on union of the interiors of the sets $A_k$ by Assumption~\ref{assumption:boundary:mass}, and therefore $\pi(\partial A_k\times\Xi)=0$ for any coupling $\pi\in\Pi(Q,P)$, we can equivalently consider the costs 
\begin{align*}
c_k(\zeta,\xi) :=\bone_{{\rm int}(A_k)\times \Xi}(\zeta,\xi)\rho(\zeta,\xi)^p=\bone_{{\rm int}(A_k)}(\zeta)\rho(\zeta,\xi)^p.
\end{align*}

We next show that these costs satisfy Assumption~\ref{assumption:cost:functions}. To this end, let  $\xi_0\in\Xi$ and pick a sufficiently large radius $r>0$ such that any set ${\rm int}(A_k)$ has nonempty intersection with ${\rm int}(B(\xi_0,r))$, where $B(\xi_0,r)$
denotes the closed ball of center $\xi_0$ and radius $r$ in $\Xi$ in the metric $\rho$. Since $\Xi$ has the subspace topology of $\Rat{d}$, we can pick for each $k$ an increasing sequence of compact sets $K_{k,\nu}$ such that 
\begin{align*}
\cup_{\nu=1}^\infty K_{k,\nu}={\rm int}(A_k).
\end{align*}
For example, we may take 
\begin{align*}
K_{k,\nu}:=\{\xi\in\Xi:\rho(\xi_0,\xi)\le\nu\;\text{and}\;\rho(\xi,\Xi\setminus{\rm int}(A_k))\ge 1/\nu\}.
\end{align*}
Then, by Urysohn's theorem (see e.g., \cite[Page 32]{VIB-OGS:20}), we can select for each $k$, $\nu$ a continuous function $\phi_{k,\nu}:\Xi\to[0,1]$ with $\phi_{k,\nu}(\xi)=1$ for all $\xi\in K_{k,\nu}$ and $\phi_{k,\nu}(\xi)=0$ for all $\xi\in \Xi\setminus{\rm int}(A_k)$. Setting $\psi_{k,\nu}=\max_{\ell\le\nu}\phi_{k,\ell}$, it follows that    
\begin{align*}
c_{k,\nu}(\zeta,\xi)=\psi_{k,\nu}(\zeta)\rho(\zeta,\xi)^p
\end{align*}
is a sequence of continuous and nondecreasing functions with 
\begin{align*}
c_{k,\nu}(\zeta,\xi)\to c_k(\zeta,\xi)(=\bone_{{\rm int}(A_k)}(\zeta)\rho(\zeta,\xi)^p),
\end{align*}
therefore satisfying Assumption~\ref{assumption:cost:functions}(i). Recalling that for the selected $r>0$ we have ${\rm int}(A_k)\cap {\rm int}(B(\xi_0,r))\ne\emptyset$, we may assume without loss of generality that also 
\begin{align*}
{\rm int}(K_{k,\nu})\cap {\rm int}(B(\xi_0,r))\ne\emptyset
\end{align*}
for all $\nu\in\mathbb N$ (by possibly neglecting a finite number of initial indices $\nu$ and starting indexing the sets $K_{k,\nu}$ from there on). This implies that the functions $c_{k,\nu}$ are linearly independent on $C(\Xi_{\rm cmp}\times \Xi_{\rm cmp})$ with $\Xi_{\rm cmp}\equiv B(\xi_0,r)$.

Indeed, assume on the contrary that $\sum_{k=1}^m a_kc_{k,\nu}\equiv 0$ with $a_k\ne 0$ for some $k\in[m]$, and let $\zeta\in {\rm int}(K_{k,\nu})\cap {\rm int}(B(\xi_0,r))$, which yields 
\begin{align} \label{ball:in:interior:intersection} 
B_{\Rat{d}}(\zeta,\epsilon)\cap\Xi\subset 
{\rm int}(K_{k,\nu})\cap {\rm int}(B(\xi_0,r))
\end{align}
for some $\epsilon>0$, where the $\Rat{d}$-subscript in this case signifies that the ball is taken with respect to $\Rat{d}$. Since the sets ${\rm int}(A_k)$ are disjoint, we have that $c_{\ell,\nu}(\zeta,\xi)=0$ for all $\xi\in B(\xi_0,r)$ when $\ell\ne k$. Thus, since $a_k\ne 0$, we necessarily get that  $c_{k,\nu}(\zeta,\xi)=0$ for all $\xi\in B(\xi_0,r)$. We claim that $B_{\Rat{d}}(\zeta,\epsilon)\cap\Xi$ does not consist of a single element, since otherwise  we would have 
$B_{\Rat{d}}(\zeta,\epsilon)\setminus\{\zeta\}\cap\Xi=\emptyset$,
%
contradicting our assumption ${\rm cl}_{\Rat{d}}({\rm int}_{\Rat{d}}(\Xi))=\Xi$. This implies by \eqref{ball:in:interior:intersection} that there is some $\xi\in{\rm int}(K_{k,\nu})\cap {\rm int}(B(\xi_0,r))$ with $\xi\ne\zeta$, and thus that $c_{k,\nu}(\zeta,\xi)\ne 0$, which is a contradiction. In a similar manner one can show that the functions $c_{k,\nu}$ cannot yield a linear combination that is constant in the second argument as a function on $\Xi_{\rm cmp}\times \Xi_{\rm cmp}$. We thus conclude that Assumption~\ref{assumption:cost:functions}(ii) is also satisfied. 

Therefore, we get from Theorem~\ref{thm:strong:duality} that 
\begin{align*}
\sup_{P\in \lcAmbSet{Q}{\bm\eps} } \mathbb{E}_P[h(\xi)] & 
= \inf_{\bm\lambda\succeq0}\Big\{ 
\langle \bm\lambda,\bm\eps \rangle  +\int_{\zeta\in \Xi}\sup_{\xi\in\Xi}\Big\{ h(\xi)-
\sum_{k=1}^m\lambda_{k}c_k(\zeta,\xi) \Big\}dQ(\zeta)\Big\} \\
& =\inf_{\bm\lambda\succeq0}\Big\{ 
\langle \bm\lambda,\bm\eps \rangle +\int_{\zeta\in \Xi}\sup_{\xi\in\Xi}\Big\{  h(\xi)-
\sum_{k=1}^m\lambda_{k}\bone_{{\rm int}(A_k)}(\zeta)\rho(\zeta,\xi)^p\Big\}dQ(\zeta)\Big\} \\
& =\inf_{\bm\lambda\succeq0}\Big\{ 
\langle \bm\lambda,\bm\eps \rangle  +\sum_{\ell=1}^m\int_{\zeta\in A_\ell}\sup_{\xi\in\Xi}\Big\{ h(\xi) -\sum_{k=1}^m\lambda_{k}\bone_{{\rm int}(A_k)}(\zeta)\rho(\zeta,\xi)^p\Big\}dQ(\zeta)\Big\}. 
\end{align*}
Taking into account that $\sum_{k=1}^m\lambda_{k}\bone_{{\rm int}(A_k)}(\zeta)=\lambda_\ell$ for each $\zeta\in {\rm int}(A_\ell)$, and that the union of the boundaries of the sets $A_\ell$ have $Q$-measure zero, we obtain the first reformulation in the statement. The second reformulation follows directly by substituting $Q\equiv P^n$ in the first one.      
\end{proof}

Next we specialize to the case where $\Xi$ is a subset of $\Rat{d}$ equipped with some norm $\|\cdot\|$, to obtain analogous convex reformulations as in Wasserstein DRO~\cite{mohajerin2018data}. In analogy to~\cite{mohajerin2018data}, we  make the following assumption. 

\begin{assumption}
\label{assumption:h_j}
\longthmtitle{Concavity of constituent functions} 
The cost function $h$ can be written as the pointwise maximum of extended real-valued functions, namely $h(\xi) = \max_{j\in[J]} h_j(\xi)$ where $h_j:\Rat{d}\to \Bar{\mathbb{R}}$. Here  $\Xi$ is convex and closed, and all the functions $-h_j$ are convex, lower-semicontinuous,  proper, and not identically $-\infty$ on $
\Xi$.
\end{assumption}

\begin{proposition} \label{prop:tractable reformulation}
\longthmtitle{Convex reformulation} 
In addition to the assumptions of Theorem~\ref{thm:dual:local:discrepancy} let also Assumption \ref{assumption:h_j} hold. Then for $p=1$, the worst case expectation \eqref{worst case expectation} with $\mathcal P\equiv\lcAmbSet{1}{\empirical}{\bm\eps}$ is equal to the value of the convex program  
    \begin{align*}
&\inf_{\bm\lambda,s_{ki},z_{kij},\nu_{kij}} \langle \bm\lambda,\bm\eps\rangle + \frac{1}{n}\sum_{k=1}^m \sum_{i=1}^{n_k} s_{ki}\\
&\hspace{2.4em}\textup{s.t.} \;[-{h_j}]^*(z_{kij}-\nu_{kij}) + \sigma_{\Xi}(\nu_{kij}) - \langle z_{kij},\xi_{ki}  \rangle \leq s_{ki}\\
&\hspace{3.8em} \|z_{kij}\|_* \leq \lambda_k \\
&\hspace{6.4em} k\in[m],i\in[n_k], j\in[J].
    \end{align*}
\end{proposition}

\begin{proof}
By introducing epigraphical
variables $s_{ki},\; k\in[m],\;i\in[n_k]$ on the right-hand side of the reformulation of Theorem~\ref{thm:dual:local:discrepancy} and following the same reasoning as in the proof of  \cite[Theorem 4.2]{mohajerin2018data} (see also \cite[Proposition 5.1]{chaouach2025tractable}) we obtain the desired result.
\end{proof}

\section{Distributionally robust location optimization }

Here, we provide tractable reformulations of a single facility optimal location problem. As a motivating example, consider the problem where we seek to optimally place a fire department in a city $\Xi\subset\Rat{2}$. In this setting, the department is responsible for a fixed part $B \subset \Xi$ of the urban area. Our goal is to place the fire department in such a way that it minimizes the average response time to incidents within its area of responsibility.
We denote by $a$ the location of the facility, where $a\in B$.
Then the optimal location problem becomes 
\begin{align*}
    \inf_{a\in B} \mathbb{E}_{P} [\bone_{B}(\xi)\| \xi-a \|_1]
\end{align*}
where incidents $\xi$ across the whole city are distributed according to $P$ and the $1$-norm between locations captures the rectangular layout of the city grid. Despite our 2D motivation, we formulate the problem for general dimensions $N$. 
We seek to solve its  distributionally robust version \begin{align}
    \inf_{a\in B} \sup_{P\in \lcAmbSet{1}{\empirical}{\bm\eps}} \mathbb{E}_P [\bone_{B}(\xi)\| \xi-a \|_1],
    \label{eq:DRO_coverage_single_agent}
\end{align}
where the distribution $P$ is unknown and lies in a local-discrepancy ambiguity set, defined over a covering $\A_m$ of~$\Xi$. 
Assuming that the set $B$ and the  space $\Xi$ are convex polyhedra, i.e., $B:= \{\xi\in\Rat{N}: C_B \xi\preceq d_B\}$ and $\Xi:=\{\xi\in\Rat{N}: C\xi\preceq d\}$, we provide a tractable reformulation of~\eqref{eq:DRO_coverage_single_agent}.

\begin{proposition}
\longthmtitle{Reformulation of robust facility location} The optimal value and location of the distributionally robust facility location problem  \eqref{eq:DRO_coverage_single_agent} are equivalently obtained by solving the optimization problem 
\begin{align*}
\inf_{\substack{a \in B, \lambda_k,s_{ki},z_{ki\eta},z_{ki}\\\gamma_{ki\eta},\delta_{ki\eta},\delta_{ki}\succeq 0}} & \langle \bm\lambda,\bm\eps\rangle + \frac{1}{n}\sum_{k=1}^m \sum_{i=1}^{n_k} s_{ki}
\\
{\rm s.t.} \hspace{2.5em}  & -\langle \eta,a \rangle + \langle   d_B,\gamma_{ki\eta}   \rangle +\langle \delta_{ki\eta},d \rangle - \langle z_{ki\eta},\xi_{ki}  \rangle \leq s_{ki}
\\ 
& \langle \delta_{ki}, d \rangle  - \langle z_{ki},\xi_{ki}  \rangle \leq s_{ki} \\
&  C_B^\top \gamma_{ki\eta} = z_{ki\eta}-C^\top \delta_{ki\eta}+\eta \\
&  C^\top \delta_{ki} = z_{ki}
\\
& \|z_{ki\eta}\|_* \leq \lambda_k ,\quad \|z_{ki}\|_* \leq \lambda_k \\
& k\in[m],\;i\in[n_k],\; \eta\in\{-1,1\}^N.
\end{align*}
\label{proposition: tractable reformulation for example}
\end{proposition}  

\begin{proof}
For the proof, we write the cost function $h$ as 
\begin{align*}
    h(\xi):= \bone_{B}(\xi)\| \xi-a \|_1 = \max \Big\{\max_{\substack{\eta\in\{-1,1\}^N}} \{h_{\eta}(\xi)\},h_0(\xi)\Big\},
\end{align*}
where $    h_\eta(\xi) := \langle \eta,\xi-a \rangle -\chi_{B}(\xi)\; {\rm{ with }}\;  \eta\in\{-1,1\}^N,$ and $h_0(\xi) = 0$.
Each $-h_\eta$, $-h_0$, is proper, lower semicontinuous, and convex; hence Assumption~\ref{assumption:h_j} holds. In addition, since $B$ is closed, $h$ is upper semicontinuous, which establishes also Assumption~\ref{assumption:h}.   
Thus, according to Proposition~\ref{prop:tractable reformulation}, the  worst-case expectation in \eqref{eq:DRO_coverage_single_agent} is equal to the value of the optimization problem  
\begin{align*}
\inf_{\substack{\lambda_k,s_{ki},z_{ki\eta},\nu_{ki\eta}\\z_{ki},\nu_{ki}}} & \langle \bm\lambda,\bm\eps\rangle + \frac{1}{n}\sum_{k=1}^m \sum_{i=1}^{n_k} s_{ki}\\
\text{s.t.}\hspace{1.9em} & [-{h_{\eta}}]^*(z_{ki\eta}-\nu_{ki\eta}) + \sigma_{\Xi}(\nu_{ki\eta}) - \langle z_{ki\eta},\xi_{ki}  \rangle \leq s_{ki}\\
&[- h_0]^*(z_{ki}-\nu_{ki}) + \sigma_{\Xi}(\nu_{ki}) - \langle z_{ki},\xi_{ki}  \rangle \leq s_{ki}\\
&\|z_{ki\eta}\|_* \leq \lambda_k, \quad \|z_{ki}\|_* \leq \lambda_k \\
&k\in[m],\;i\in[n_k],\;\eta\in \{-1,1\}^N.
\end{align*}

Here, the convex conjugate of each $-h_{\eta}$ is  
\begin{align*}
    [-h_{\eta}]^*(z) &= \sup_{\xi\in\Rat{N}} \langle z,\xi \rangle + h_{\eta}(\xi) \\&= \sup_{\xi\in\Rat{N}} \langle z,\xi \rangle + \langle \eta,\xi-a \rangle - \chi_{B}(\xi)
     \\&= \sup_{\xi\in B} \langle z,\xi \rangle + \langle \eta,\xi-a \rangle
     \\&= -\langle \eta,a \rangle +\sup_{C_B\xi\preceq d_B} \langle z +  \eta,\xi \rangle
     \\&= -\langle \eta,a \rangle +\inf_{\substack{(C_B)^\top \gamma = z
     +\eta\\\gamma\succeq0}}\langle   d_{B},\gamma   \rangle.
\end{align*}
Taking further into account that $[- h_0]^*(z)=\chi_{\{0\}}(z)$ and computing the support function 
\begin{align*}
\sigma_{\Xi}(\nu) 
= \sup_{\xi\in\Xi} \ \langle \nu,\xi\rangle = \sup_{C\xi \preceq d} \ \langle \nu,\xi\rangle 
= \inf_{\substack{C^\top \delta = \nu\\\delta \succeq 0}} \ \langle \delta,d\rangle,
\end{align*}
we may plug all these expressions in the constraints of the above program (removing the ``inf" operators as they are substituted in the left-hand-side of the corresponding inequalities) to obtain the equivalent reformulation 
\begin{align*}
\inf_{\substack{\lambda_k,s_{ki},z_{ki\eta},\nu_{ki\eta},z_{ki},\nu_{ki}\\\gamma_{ki\eta},\delta_{ki\eta},\delta_{ki}\;\succeq 0}} & \langle \bm\lambda,\bm\eps\rangle + \frac{1}{n}\sum_{k=1}^m \sum_{i=1}^{n_k} s_{ki}\\
\text{s.t.}\hspace{3.4em} & -\langle \eta,a \rangle + \langle   d_{B},\gamma_{ki\eta}   \rangle +\langle \delta_{ki\eta},d \rangle  -\langle z_{ki\eta},\xi_{ki}  \rangle \leq s_{ki} \\ 
& \chi_{\{0\}}(z_{ki} - \nu_{ki}) +  \langle \delta_{ki}, d \rangle -\langle z_{ki},\xi_{ki}  \rangle \leq s_{ki} \\
& C_B^\top \gamma_{ki\eta} = z_{ki\eta}-\nu_{ki\eta}+\eta \\
& C^\top \delta_{ki\eta} = \nu_{ki\eta},\quad C^\top \delta_{ki} = \nu_{ki} \\
& \|z_{ki\eta}\|_* \leq \lambda_k, \quad \|z_{ki}\|_* \leq \lambda_k\\
& k\in[m],\;i\in[n_k],\; \eta\in\{-1,1\}^N.
\end{align*}
Since the term $\chi_{\{0\}}(z_{ki} - \nu_{ki})$ enforces the additional equality constraint $z_{ki} = \nu_{ki}$,  minimizing the above expression over the location $a\in B$ yields the desired result.  
\end{proof}

\section{Data-driven tuning of local transport costs}

\label{sec:Data-driven tuning of local transport costs}

Among our goals is to compare the 
performance of the ambiguity sets $\Ball{\empirical}{\eps}$ and $\lcAmbSet{\empirical}{\bm\eps}$. To achieve this, we seek a data-driven calibration process to determine their radii $\eps$ and $\bm\eps$ in a consistent manner. 
Although the underlying distribution $P$ is considered unknown, we assume that it belongs to a known class $\class$, and introduce a computational process to approximate it when the class is parametric, i.e., when $\class := \{P_\theta\}_{\theta\in\Theta}\subset\P(\Xi)$ for some finite-dimensional parameter set $\Theta\subset\Rat{\ell}$.
Our ideal goal is to pick radii $\eps$ and $\bm{\eps}$ that are  ``as small as possible" and ensure that the corresponding ambiguity sets contain the distribution $P$ with prescribed confidence $1-\beta$. Since the distribution $P$ in our setting is unknown, it can be any element of the class $\class$.

Selecting the smallest possible radius of the monolithic ambiguity ball $\Ball{\empirical}{\eps}$, which contains any data-generating distribution $P\in\mathcal C$ with prescribed confidence, is a clear task, at least conceptually. Namely, for any fixed $\beta\in(0,1)$ and number $n$ of samples, it suffices to select    
\begin{align}
& \eps\equiv\eps(n,\beta;\mathcal C):=\inf\{\epsilon>0: \bP(P\in\B_p(P^n,\epsilon)) \ge 1-\beta\;\textup{for all}\;P\in\mathcal C\}.  \label{eq:optimal radius scalar}
\end{align}

On the other hand, it is not possible to select the smallest vector-valued radius $\bm\eps$ in a similar manner, since $\RgeO^{m}$ is not totally ordered. To overcome this obstacle, we consider an auxiliary scalar discrepancy measure $h:\RgeO^{m}\to\RgeO$ and seek the vector radii that minimize $h(\bm\eps)$. 
Choosing the strongly convex function $h(x)=\|x\|_2^2$ enables us to select a unique vector $\bm\eps$ from the closure of the set of feasible vector discrepancies. In particular, we set  
\begin{align*}
\mathscr E(Q,P):=\argmin_{\bm\eps\in{\rm cl}(T_{p,\A_m}(Q,P))}\|\bm\eps\|_2^2,    
\end{align*}
which is uniquely defined. Since we are interested in a data-driven tuning of the vector radius, we define
\begin{align*}
\bm v(n;P):=\frac{\bE[\mathscr E(P^n,P)]}{\|\bE[\mathscr E(P^n,P)]\|_2}\in\mathbb S^{m-1}\cap\RgeO^m
\end{align*}
and set 
\begin{align*}
\bm v(n;\mathcal C):=\argmin_{\bm v\in\mathbb S^{m-1}}r(\bm v)
\end{align*}
where 
\begin{align*}
r(\bm v):=\inf\{r>0:\bm v(n;P)\in B(\bm v,r)\;\textup{for all}\;P\in\mathcal C\}.
\end{align*}
In this setting, $\bm v(n;P)$ denotes the expected direction of the uniquely selected discrepancy vector $\mathscr E(P^n,P)$ between the empirical and the true distribution. Then the purpose of selecting $\bm v(n;\mathcal C)$ as above is to find a point on the sphere $\mathbb S^{m-1}$ that  minimizes the worst-case distance from any expected direction generated by the distributions in $\mathcal C$. This point will be the direction along which we select the vector-valued radius of the data-driven  ambiguity set for distributions in class $\mathcal C$. Thus, in analogy to the data-driven selection of $\eps$ for Wasserstein balls, we select
\begin{align}
\bm\eps \equiv\; & \bm\eps(n,\beta;\mathcal C):=\lambda_\star(n,\beta;\mathcal C)\bm v(n;\mathcal C) \label{eq:optimal:vector:radius} \\
\lambda_\star(n,\beta;\mathcal C):=\; & \inf\{\lambda>0: \bP(P\in\lcAmbSet{P^n}{\lambda \bm v(n;\mathcal C))} \ge 1-\beta\;\textup{for all}\;P\in\mathcal C\}.  \nonumber
\end{align}

We next outline how radii such as the one in \eqref{eq:optimal radius scalar} can be approximated computationally in low-dimensional parametric settings with compact parameter sets for fixed $n$ and $\beta$. The idea relies on the following five-step process: 
\def\theenumi{\roman{enumi}}
\begin{enumerate}
    \item sample a large number $\ell\in\mathbb N$ of distributions $P_{\theta_\kappa}$ from the class $\mathcal C$ (or preselect a grid of parameter values); 
    \item for each such distribution $P_{\theta_\kappa}$ sample a large number  $\nu\in\mathbb N$ of empirical distributions $P_{\theta_\kappa}^{n,j}$, $j\in[\nu]$;
    \item for each $\kappa\in[\ell]$ and $j\in[\nu]$ compute the Wasserstein distance $\eps^{\kappa,j}:=W_p(P_{\theta_\kappa}^{n,j},P_{\theta_\kappa})$;
    \item for each sampled distribution $P_{\theta_\kappa}$ pick
    \begin{align*}
    \eps^\kappa:=\min_{j\in[\nu]}\Big\{\eps^{\kappa,j}:\frac{\#\{j'\in[\nu]:\eps^{\kappa,j'}\le\eps^{\kappa,j}\}}{\nu}\ge1-\beta\Big\},
    \end{align*} 
    which yields an empirical lower value for the radius of the ball that contains $P_{\theta_\kappa}$ with confidence $1-\beta$;
    \item select $\widehat\eps:=\max_{\kappa\in[\ell]}\eps^\kappa$ as an empirical approximation of the radius $\eps$ in \eqref{eq:optimal radius scalar}.  
\end{enumerate}

We next outline how these steps can be modified to obtain a data-driven approximation of the vector-valued radius $\bm\eps(n,\beta;\mathcal C)$ in \eqref{eq:optimal:vector:radius}. To this end, we will exploit the fact that the vector-valued discrepancy $\mathscr E(Q,P)$ can be determined by the optimizer $\bm\eps_{\star}=(\eps_{1\star},\ldots,\eps_{m\star})$ of the  problem  
\begin{align}
& \inf_{\bm\eps\in\RgeO^m,\pi\in \Pi(Q,P)} \|\bm\eps\|_2^2 \label{minimum:norm:discrepancy} \\
&\qquad\quad\;{\rm s.t.}\int_{A_k \times \Xi} \rho(\zeta,\xi)^p d\pi(\zeta,\xi) \leq  \eps_k\;\textup{for all}\;k\in[m]. \nonumber 
\end{align}

\begin{rem}
\longthmtitle{Existence of minimizer}
By the results provided so far, there is no guarantee that the optimal solution of the problem~\eqref{minimum:norm:discrepancy} is attained for some feasible pair $(\bm\eps,\pi)$. It can be rather understood as the limit in the first argument of a feasible sequence $(\bm\eps^\kappa,\pi^\kappa)$. Nevertheless, by convexity of the set of feasible $(\bm\eps,\pi)$ and strong convexity of $\bm\eps\mapsto\|\bm\eps\|_2^2$ this limit is guaranteed to be unique. \hfill $\Box$
\end{rem}

Steps (i) and (ii) remain exactly the same for the data-driven tuning of the vector radius $\bm\eps$, while Steps (iii) and (iv) are substantially modified and involve more computations. In particular, we have the updated steps:     

\begin{itemize}
\item [(iiia)] for each $\kappa\in[\ell]$ and $j\in[\nu]$ compute the minimum-norm vector discrepancy $\bm\eps^{\kappa,j}:=\mathscr E(P_{\theta_\kappa}^{n,j},P_{\theta_\kappa})$ via the quadratic optimization problem~\eqref{minimum:norm:discrepancy};
\item [(iiib)] for each $\kappa\in[\ell]$  approximate the expected value of $\mathscr E(P_{\theta_\kappa}^n,P_{\theta_\kappa})$ by $\bm\eps^\kappa:=\frac{1}{\nu}\sum_{j\in[\nu]}\bm\eps^{\kappa,j}$
 and approximate $\bm v^\kappa(n;P_{\theta_\kappa})$ by the data-driven vector $\bm v^\kappa:=\bm\eps^\kappa/\|\bm\eps^\kappa\|_2$;  
 \item [(iiic)] approximate $\bm v(n;\mathcal C)$ by $\widehat{\bm v}:=\argmin_{\bm v\in\mathbb S^{m-1}}\widehat r(\bm v)$ 
with  
\begin{align*}
\widehat r(\bm v):=\inf\big\{r>0:\bm v^\kappa(n;P_{\theta_\kappa})\in B(\bm v,r) \; \textup{for all}\;\kappa\in[\ell]\};
\end{align*} 
 \item [(iiid)] for each  $\kappa\in[\ell]$ and $j\in[\nu]$ compute the scaling $\lambda_\star^{\kappa,j}:=\lambda_\star(\widehat{\bm v};P_{\theta_\kappa}^{n,j},P_{\theta_\kappa})$  of the radii of the minimum size ambiguity set in the direction $\widehat{\bm v}$ by solving the linear program of Corollary~\ref{corollary:minimum:size};
\item [(iv)] for each sampled distribution $P_{\theta_\kappa}$ pick
  \begin{align*}
  \lambda_\star^\kappa:=\min_{j\in[\nu]}\Big\{\lambda_\star^{\kappa,j}:\frac{\#\{j'\in[\nu]:\lambda_\star^{\kappa,j'}\le\lambda_\star^{\kappa,j}\}}{\nu}\ge1-\beta\Big\};
    \end{align*} 
 \item [(v)] select $\widehat{\bm\eps}:=\max_{\kappa\in[\ell]}\lambda_\star^\kappa\widehat{\bm v}$ as an empirical approximation of the radius $\bm\eps$ in \eqref{eq:optimal:vector:radius}.
\end{itemize}

\begin{rem}
\longthmtitle{Numerical optimization  aspects} 
The approximation of the Wasserstein distances in Step (iii) of the outlined numerical process for the monolithic ambiguity radius can be carried out efficiently using tools from computational optimal transport~\cite{GP-MC:19}.
In direct analogy to the classical optimal transport setting, the optimization problems in both \eqref{minimum:norm:discrepancy} and Corollary~\ref{corollary:minimum:size} (used in Steps (iiia) \& (iiid)) are convex programs with linear constraints on the space of couplings and can be well approximated by finite-dimensional convex programs (e.g., by quantizing  the  distributions). \hfill $\Box$
\end{rem}

\section{Simulations}

In this section, we provide an illustrative simulation example to show how local-discrepancy ambiguity sets can outperform monolithic ambiguity balls for worst-case expectations that are affected by restricted regions of the uncertainty domain. We consider the facility location problem depicted in Figure~\ref{fig:Setup}.
The objective is to determine the optimal locations of two fire stations, each responsible for one of the two red rectangular regions.
For each responsibility region, we compute the optimal location by solving \eqref{eq:DRO_coverage_single_agent} under the local-discrepancy ambiguity set $\mathcal{T}_{1,\A_3}(P^n,\bm\eps)$ and under the monolithic Wasserstein ball $\B_{1}(P^n,\eps)$. The partition of the sample space and the true underlying distribution $P_{\rm{true} }$ are shown in Figure~\ref{fig:Setup}.
\begin{figure}[h]
	\centering
	\includegraphics[width=1\linewidth]{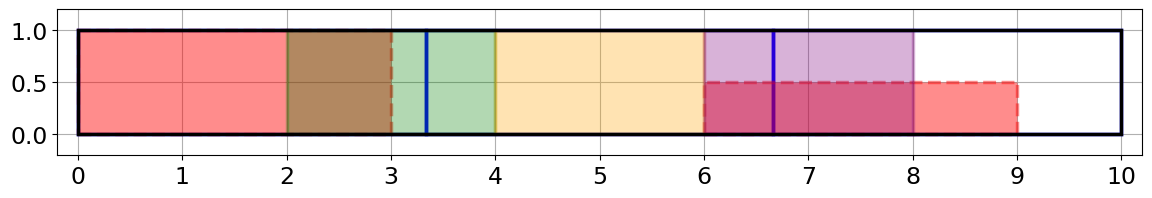}  
	\caption{The sample space $\Xi = [0,10]\times [0, 1]$ is covered by the three areas $\A_3$ between the blue segments, the true underlying distribution is the three-component mixture of uniform distributions $1/3(\mathcal U([2,4]\times [0,1])+\mathcal U([4,6]\times [0,1])+\mathcal U([6,8]\times [0,1])$ (uniforms over green, yellow and purple rectangles respectively), and the red rectangles denote the areas of responsibility $B_1({\rm{left}}), B_2({\rm{right}})$ of two fire departments.
    } \label{fig:Setup}
\end{figure} 

To assess the statistical performance of the two ambiguity sets, we generated $100$ independent empirical measures $P^n$, each based on $100$ samples, and computed the corresponding optimal DRO facility locations.
The radii $\bm\eps,\eps$ were calculated according to the process described in Section~\ref{sec:Data-driven tuning of local transport costs}, where $\mathcal{C}$ consists of all convex combinations of the three uniform distributions $\mathcal U([2,4]\times [0,1])$, $\mathcal U([4,6]\times [0,1])$, and $\mathcal U([6,8]\times [0,1])$, with weights $m_1,m_2,m_3\in[1/3-0.1,1/3+0.1]$. We selected a confidence level of $70\%$, which yielded the radii $\eps = 0.319$ and $\bm\eps=(0.1,0.195,0.109)$ for the ball and the local-discrepancy ambiguity set, respectively.
For each empirical measure, we solved the DRO problem using both ambiguity sets and evaluated the true expected cost at the resulting optimizers. The local-discrepancy ambiguity set achieved a lower expected cost in $93\%$ of the trials for region $B_1$ and $68\%$ for region $B_2$, as shown in Figure~\ref{fig:box-plots}. Overall, the local-discrepancy ambiguity set shows a clear advantage over the monolithic Wasserstein ball, yielding lower true expected costs substantially more often in both facility-location problems.

It is worth noting that the statistical performance of the two ambiguity sets is assessed across multiple optimization problems affected by the same uncertainty. The proposed setting also retains information from samples outside the responsibility regions, unlike restricting the empirical measure to $B_1$ and $B_2$; nearby samples may still provide information about probability mass that affects the decision.

\begin{figure}[h]
	\centering
	\includegraphics[width=0.75\linewidth]{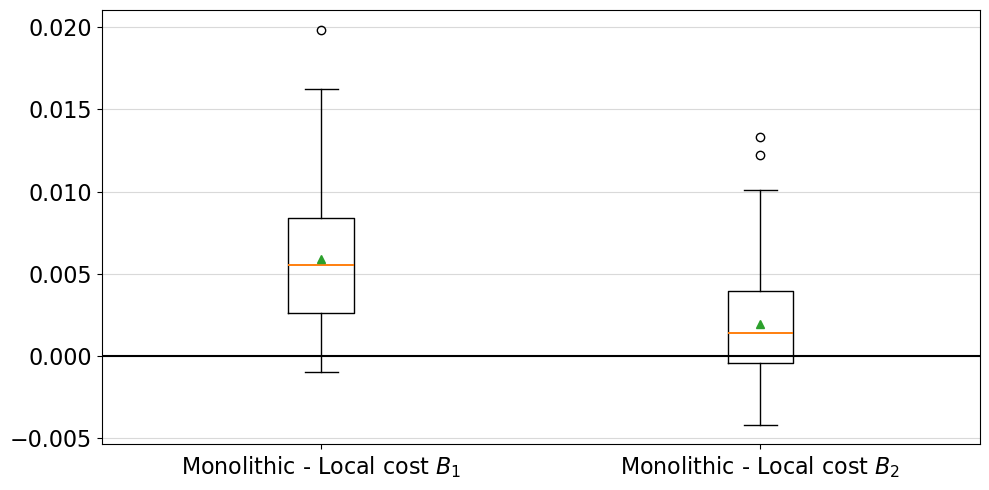}  
	\caption{The figure shows the difference between the true expected cost using the DRO optimizer of the Wasserstein ball and the corresponding local-discrepancy ambiguity set over $100$ trials. For both regions,  the local-discrepancy ambiguity set exhibits decisions with superior statistical performance.    
    } \label{fig:box-plots}
\end{figure}

\section{Conclusion}

This paper introduced local-discrepancy ambiguity sets as an alternative to monolithic Wasserstein balls. These sets rely on multiple local transport constraints across sub-regions of the uncertainty domain. We derived strong duality results and tractable reformulations for DRO problems over these ambiguity sets and applied them to optimal facility location problems. The results demonstrate a clear performance advantage of the proposed ambiguity sets in the considered facility-location problems, where the decisions depend on uncertainty over restricted regions of the domain. Future work includes the derivation of finite-sample guarantees for tuning the vector-valued radii and their use to establish conditions under which local-discrepancy ambiguity sets outperform Wasserstein balls in suitable DRO and uncertainty-quantification problems. We will also explore optimal partitioning schemes of the sample space that may depend on the data and on prior knowledge of the distribution class.

\bibliographystyle{IEEEtran}
\bibliography{Ref}
\end{document}